\documentclass[12pt,oneside,english]{amsart}
\usepackage[T1]{fontenc}
\usepackage[latin9]{inputenc}
\usepackage{geometry}
\usepackage{float}
\usepackage{mathrsfs}
\usepackage{amstext}
\usepackage{amsthm}
\usepackage{amssymb}
\usepackage{comment}

\makeatletter
\numberwithin{equation}{section}
\numberwithin{figure}{section}
\theoremstyle{plain}
\newtheorem{thm}{\protect\theoremname}[section]
\theoremstyle{remark}
\newtheorem{rem}[thm]{\protect\remarkname}
\theoremstyle{plain}
\newtheorem{lem}[thm]{\protect\lemmaname}
\theoremstyle{plain}
\newtheorem{prop}[thm]{\protect\propositionname}
\theoremstyle{definition}

\usepackage[svgnames]{xcolor}
\usepackage[bookmarksnumbered=true]{hyperref} 
\hypersetup{
	colorlinks = true,
	linkcolor = Blue,
	anchorcolor = blue,
	citecolor = Green,
	filecolor = blue,
	urlcolor = FireBrick
}

\makeatother

\usepackage{babel}
\providecommand{\examplename}{Example}
\providecommand{\lemmaname}{Lemma}
\providecommand{\propositionname}{Proposition}
\providecommand{\remarkname}{Remark}
\providecommand{\theoremname}{Theorem}
\newcommand{\R}{{\mathbb R}}

\begin{document}
\title{Landis-type conjecture for the half-Laplacian}
\author{Pu-Zhao Kow}
\address{Department of Mathematics and Statistics, P.O. Box 35 (MaD), FI-40014 University of Jyv\"{a}skyl\"{a}, Finland.}
\email{pu-zhao.pz.kow@jyu.fi}
\author{Jenn-Nan Wang}
\address{Institute of Applied Mathematical Sciences, NCTS, National Taiwan University, Taipei 106, Taiwan.}
\email{jnwang@math.ntu.edu.tw}
\begin{abstract}
In this paper, we study the Landis-type conjecture, i.e., unique continuation
property from infinity, of the fractional Schr\"{o}dinger equation
with drift and potential terms. We show that if any solution of the equation
decays at a certain exponential rate, then it must be trivial. The
main ingredients of our proof are the Caffarelli-Silvestre extension
and Armitage's Liouville-type theorem. 
\end{abstract}

\maketitle

\begin{sloppypar}

\section{Introduction}

In this paper, we consider the following equation with the half Laplacian
\begin{equation}
(-\Delta)^{\frac{1}{2}}u+{\bf b}({\bf x})\cdot\nabla u+q({\bf x})u=0\quad\text{in}\;\;\mathbb{R}^{n},\label{eq:Schrodinger}
\end{equation}
where $n\ge 1$. Our aim is to investigate the minimal decay rate of nontrivial solutions of \eqref{eq:Schrodinger}. In other words, we consider
the unique continuation property from infinity of \eqref{eq:Schrodinger}.
This problem is closely related to the conjecture proposed by Landis
in the 60's \cite{KL88}. Landis conjectured that, if $u$ is a solution
to the classical Schr\"{o}dinger equation
\begin{equation}
-\Delta u+q({\bf x})u=0\quad\text{in}\;\;\mathbb{R}^{n},\label{eq:Schrodinger-classical}
\end{equation}
with a bounded potential $q$, satisfying the decay estimate 
\[
|u({\bf x})|\le\exp(-C|{\bf x}|^{1+}),
\]
then $u\equiv0$. Landis' conjecture was disproved by Meshkov \cite{Meshkov92}, who constructed a \emph{complex-valued}
potential $q\in L^{\infty}(\mathbb{R}^{n})$ and a nontrivial solution
$u$ of \eqref{eq:Schrodinger-classical} such that 
\[
|u({\bf x})|\le\exp(-C|{\bf x}|^{\frac{4}{3}}).
\]
In the same work, Meshkov showed that if
\begin{equation*}
|u({\bf x})|\le\exp(-C|{\bf x}|^{\frac{4}{3}+}),
\end{equation*}
then $u\equiv 0$. Based on a suitable Carleman estimate, a quantitative version of Meshkov's result was established in \cite{BK05quantitativeLandis}, see also \cite{Cruz-Sampedro99,Davey14magneticSch,DZ18Landis,DZ19Landis,KL19LandisNS,LUW11LandisNS,LW14Landis} for related results. We also refer to \cite[Theorem~2]{Zhu18LandisHighOrder} for some decay estimates at infinity for higher order elliptic equations.

In view of Meshkov's example, Kenig modified Landis' original conjecture and asked that whether the Landis'
conjecture holds true for \emph{real-valued} potentials $q$ in \cite{kenig06realLandis}.
The real version of Landis' conjecture in the plane was resolved
recently in \cite{LMNNrealLandis}. We also refer to \cite{Davey20realLandis,DKW17realLandis,DKW20realLandis,KSW15realLandis} for the early development of the real version of Landis' conjecture. 

For the fractional Schr\"{o}dinger equation, the Landis-type conjecture was studied in \cite{RW19Landis}. The main theme of this paper
is to extend the results in \cite{RW19Landis} to the fractional Schr\"{o}dinger equation with the half Laplacian \eqref{eq:Schrodinger}. Previously, the authors in \cite{KWLandisDrift} proved some partial results for the fractional Schr\"odinger equation
\begin{equation}
((-\Delta)^{s}+b({\bf x}){\bf x}\cdot\nabla + q({\bf x}))u=0\quad\text{in}\;\;\mathbb{R}^{n},\label{eq:sch}
\end{equation}
where $s\in(0,1)$ and $b$, $q$ are scalar-valued functions.
The main tools used in \cite{KWLandisDrift} are the Caffarelli-Silvestre extension and the Carleman estimate. The particular form of the drift coefficient in \eqref{eq:sch}
is due to the applicability of the Carleman estimate. It turns out when $s=\frac 12$, i.e., the case of half Laplacian, we can treat a general vector-valued drift coefficient ${\bf b}({\bf x})$ in \eqref{eq:Schrodinger}. The underlying reason is that the Caffarelli-Silvestre extension solution of $(-\Delta)^{\frac{1}{2}}u=0$ in $\R^n$ is a harmonic function in  ${\mathbb R}_+^{n+1}$.  Inspired by this observation, we show
that if both ${\bf b}$ and $q$ are differentiable, then any nontrivial
solution of \eqref{eq:Schrodinger} can not decay exponentially
at infinity. The detailed statement is described in the following
theorem. 
\begin{thm}
\label{thm:main1} 
\begin{subequations}
Assume that there exists a constant $\Lambda>0$ such that 
\begin{equation}
\|q\|_{L^{\infty}(\mathbb{R}^{n})}+\|\nabla q\|_{L^{\infty}(\mathbb{R}^{n})}+\|\nabla{\bf b}\|_{L^{\infty}(\mathbb{R}^{n})}\le\Lambda\label{eq:assumption-1}
\end{equation}
and, furthermore, there exists an $\epsilon>0$, depending only on $n$, such that 
\begin{equation}
\|{\bf b}\|_{L^{\infty}(\mathbb{R}^{n})}\le\epsilon.\label{eq:assumption-2}
\end{equation}
Let $u \in W^{2,p}(\mathbb{R}^{n})$ for some integer $p>n$ be a solution to \eqref{eq:Schrodinger} such that 
\begin{equation}
|u({\bf x})|\le\Lambda e^{-\lambda|{\bf x}|}\label{eq:decay-assumption}
\end{equation}
for some $\lambda>0$, then $u\equiv0$. 
\end{subequations}
\end{thm}

\begin{rem}
Note that both $(-\Delta)^{\frac{1}{2}}u$
and $\nabla u$ are first orders. In view of the $L^{p}$ estimate of the Riesz transform \eqref{eq:boundedness-Riesz-transform1}, \eqref{eq:boundedness-Riesz-transform2},
the assumption \eqref{eq:assumption-2} and the regularity requirement of $u$ are imposed to ensure that the non-local operator $(-\Delta)^{\frac{1}{2}}u$ is the dominated term in \eqref{eq:Schrodinger}.
\end{rem}

It is interesting to compare Theorem~\ref{thm:main1} with \cite[Theorem 1]{RW19Landis}.
Assume that $u\in H^{s}(\mathbb{R}^{n})$ is a solution to 
\begin{equation}
(-\Delta)^{s}u+q({\bf x})u=0\quad\text{in}\,\,\mathbb{R}^{n}\label{eq:fractional-Sch}
\end{equation}
such that $|q({\bf x})|\le1$ and $|{\bf x}\cdot\nabla q({\bf x})|\le1$. If
\begin{equation}\label{tx}
\int_{\mathbb{R}^{n}}e^{|{\bf x}|^{\alpha}}|u|^{2}\,\mathsf{d}{\bf x}<\infty\quad\text{for some}\,\,\alpha>1,
\end{equation}
then $u\equiv0$. Therefore, for $s=\frac 12$,  Theorem~\ref{thm:main1} extends
their results by slightly relaxing the condition on $q$ and also adding a drift term. Another key improvement is that the exponential decay rate $e^{-\lambda|{\bf x}|}$ is sharper than \eqref{tx}.

The proof of Theorem~\ref{thm:main1} consists of two steps. Inspired by \cite{RW19Landis}, we first pass the boundary decay \eqref{eq:decay-assumption} to the
bulk decay of the Caffarelli-Silvestre extension solution (harmonic function) in the extended
space $\mathbb{R}^{n}\times(0,\infty)$. In the second step, we apply the Liouville-type
theorem (Theorem~\ref{thm:Liouville}) to the harmonic function. It is noted that we do not use any Carleman
estimate here. On the other hand, using the harmonic function in the unit ball ${v}_{0}({\bf z}):=\Re(e^{-1/{\bf z}^{\alpha}})$, ${\bf z}\in{\mathbb C}$, $0<\alpha<1$  (see \cite{Jin93UCP}), it is not difficult to construct an example to show the optimality of the
Liouville-type theorem. In view of this example, we believe that the decay assumption \eqref{eq:decay-assumption} is optimal. 

When ${\bf b}\equiv 0$, the following theorem can be found in \cite[Theorem 1.1.9]{Kow21dissertation},
which was obtained using similar ideas as in the proof of Theorem~\ref{thm:main1}. 
\begin{thm}
\label{thm:main2}Let $q\in L^{\infty}(\mathbb{R}^{n})$
(not necessarily differentiable) satisfy 
\[
\|q\|_{L^{\infty}(\mathbb{R}^{n})}\le\Lambda.
\]
If $u \in H^{\frac{1}{2}}(\mathbb{R}^{n})$ is a solution to \eqref{eq:Schrodinger}
with ${\bf b}\equiv0$ such that 
\begin{equation}\label{cx}
\int_{\mathbb{R}^{n}}e^{|{\bf x}|}|u|^{2}\,\mathsf{d}{\bf x}<\infty,
\end{equation}
then $u\equiv0$. 
\end{thm}

\begin{rem}
Theorem~\ref{thm:main2} is an immediate consequence of \cite[Proposition~2.2]{RW19Landis} and Theorem~\ref{thm:Liouville} (without using Proposition~\ref{prop:boundary-bulk}). Therefore we only need \eqref{cx}, namely, \eqref{eq:decay-assumption} is unnecessary when ${\bf b} \equiv 0$.
\end{rem}

It is interesting to compare this result with \cite[Theorem 2]{RW19Landis}. There, it was proved that
if $u\in H^{s}(\mathbb{R}^{n})$ solves \eqref{eq:fractional-Sch} with $|q(x)|\le 1$ and 
\begin{equation}\label{txx}
\int_{\mathbb{R}^{n}}e^{|{\bf x}|^{\alpha}}|u|^{2}\,\mathsf{d}{\bf x}<\infty\quad\text{for some}\,\,\alpha>\frac{4s}{4s-1},
\end{equation}
then $u\equiv0$. When $s=\frac 12$, \eqref{txx} becomes
\[
\int_{\mathbb{R}^{n}}e^{|{\bf x}|^{\alpha}}|u|^{2}\,\mathsf{d}{\bf x}<\infty
\]
for $\alpha >2$, which is clearly stronger than \eqref{cx}. On the other hand, Theorem~\ref{thm:main2} holds regardless whether $q$ is real-valued or complex-valued.

This paper is organized as follows. In Section~\ref{sec:Decay-gradient},
we will study the decaying behavior of $\nabla u$. In Section~\ref{sec:Caffarelli-Silvestre-extension},
we localize the nonlocal operator $(-\Delta)^{\frac{1}{2}}$ by
the Caffarelli-Silvestre extension. In Section~\ref{sec:Some-estimates-CS},
we derive some useful estimates about the Caffarelli-Silvestre
extension $\tilde{u}$ of the solution $u$, which is harmonic. In Section~\ref{sec:Boundary-Bulk},
we obtain the decay rate of $\tilde{u}$ from that of $u$. Finally, 
we prove Theorem~\ref{thm:main1} in Section~\ref{sec:main} by Armitage's Liouville-type theorem. Furthermore, we provide another proof
of this Liouville-type theorem in Appendix~\ref{sec:Appendix1}.

\section{\label{sec:Decay-gradient}Decay of the gradient}

Let $1<p<\infty$. For each $u\in L^{p}(\mathbb{R}^{n})$, let
$\psi$ satsify $(-\Delta)^{\frac{1}{2}}\psi=u$ and let ${\bf u}:=\nabla\psi$.
Using the $L^{p}$-boundedness of the Riesz transform \cite{stein2016singular}
(see also \cite{BG13Riesz}), we can show that 
\begin{equation}
\|{\bf u}\|_{L^{p}(\mathbb{R}^{n})}\le C(n,p)\|u\|_{L^{p}(\mathbb{R}^{n})}.\label{eq:boundedness-Riesz-transform1}
\end{equation}
We remark that this estimate is also used in the proof of \cite[Theorem 2.1]{CCW01criticalQG}.  Note that we can formally write ${\bf u}=\nabla(-\Delta)^{-\frac{1}{2}}u$. Plugging $(-\Delta)^{\frac{1}{2}}\psi=u$ and ${\bf u}=\nabla\psi$
into \eqref{eq:boundedness-Riesz-transform1} implies 
\begin{equation}
\|\nabla\psi\|_{L^{p}(\mathbb{R}^{n})}\le C(n,p)\|(-\Delta)^{\frac{1}{2}}\psi\|_{L^{p}(\mathbb{R}^{n})}.\label{eq:boundedness-Riesz-transform2}
\end{equation}
Thanks to \eqref{eq:boundedness-Riesz-transform2}, we can obtain the following lemma. 
\begin{lem}
\label{lem:gradient-decay}Let $2 \le p < \infty$ be an integer. Assume that \eqref{eq:assumption-1} and \eqref{eq:assumption-2} hold. 
Let $u \in W^{2,p}(\mathbb{R}^{n})$ be a solution to \eqref{eq:Schrodinger} such that the decay assumption \eqref{eq:decay-assumption} holds,
then $(-\Delta)^{\frac{1}{2}}u$ satisfies 
\begin{equation}
\int_{\mathbb{R}^{n}}e^{\frac{\lambda}{2}|{\bf x}|}|(-\Delta)^{\frac{1}{2}}u|^{2}\,\mathsf{d}{\bf x}+\int_{\mathbb{R}^{n}}e^{\frac{\lambda}{2}|{\bf x}|} |(-\Delta)^{\frac{1}{2}}u|^{p} \,\mathsf{d}{\bf x}\le C\label{eq:gradient-decay}
\end{equation}
for some positive constant $C=C(n,p,\lambda,\Lambda)$. 
\end{lem}

\begin{proof}
We first estimate the $L^{p}$-norm of $\nabla u$. Taking $L^{p}$-norm on \eqref{eq:Schrodinger} and using \eqref{eq:boundedness-Riesz-transform2}, \eqref{eq:assumption-2}, \eqref{eq:decay-assumption}, we have
\begin{align*}
\|\nabla u\|_{L^{p}(\mathbb{R}^{n})} & \le C(n,p)\|(-\Delta)^{\frac{1}{2}}u\|_{L^{p}(\mathbb{R}^{n})}\\
 & \le C(n,p)\bigg( \|{\bf b}\|_{L^{\infty}(\mathbb{R}^{n})} \|\nabla u\|_{L^{p}(\mathbb{R}^{n})} + \|q\|_{L^{\infty}(\mathbb{R}^{n})}\|u\|_{L^{p}(\mathbb{R}^{n})}\bigg)\\
 & \le\epsilon C(n,p) \|\nabla u\|_{L^{p}(\mathbb{R}^{n})} + C(n,p,\lambda,\Lambda).
\end{align*}
Choosing $\epsilon=(2C(n,p))^{-1}$ in the estimate above gives
\begin{equation}
\|\nabla u\|_{L^{p}(\mathbb{R}^{n})}\le C(n,p,\lambda,\Lambda).\label{eq:Lp-gradient-bound}
\end{equation}
 
Next, we estimate the $L^{p}$-norm of $\nabla^{2}u$. Differentiating \eqref{eq:Schrodinger} yields
\begin{equation}
(-\Delta)^{\frac{1}{2}}\partial_{j}u+{\bf b}({\bf x})\cdot\nabla(\partial_{j}u)+\partial_{j}{\bf b}({\bf x})\cdot\nabla u+q({\bf x})\partial_{j}u+\partial_{j}q({\bf x})u=0\label{eq:Schrodinger-differentiate}
\end{equation}
for each $j=1,\cdots,n$. Taking the $L^{p}$-norm of \eqref{eq:Schrodinger-differentiate},
we have 
\begin{align*}
\|\nabla(\partial_{j}u)\|_{L^{p}(\mathbb{R}^{n})} & \le C(n,p)\|(-\Delta)^{\frac{1}{2}}\partial_{j}u\|_{L^{p}(\mathbb{R}^{n})}\\
 & \le C(n,p)\bigg( \|{\bf b}\|_{L^{\infty}(\mathbb{R}^{n})}\|\nabla(\partial_{j}u)\|_{L^{p}(\mathbb{R}^{n})} + \|\nabla{\bf b}\|_{L^{\infty}(\mathbb{R}^{n})}\|\nabla u\|_{L^{p}(\mathbb{R}^{n})}\\
 & \quad+\|q\|_{L^{\infty}(\mathbb{R}^{n})}\|\nabla u\|_{L^{p}(\mathbb{R}^{n})} + \|\nabla q\|_{L^{\infty}(\mathbb{R}^{n})} \|u\|_{L^{p}(\mathbb{R}^{n})}\bigg)\\
 & \le C(n,p) \bigg(\epsilon \|\nabla(\partial_{j}u)\|_{L^{p}(\mathbb{R}^{n})} + \Lambda \|\nabla u\|_{L^{p}(\mathbb{R}^{n})} + \Lambda \|u\|_{L^{p}(\mathbb{R}^{n})}\bigg)\\
 & \le\frac{1}{2} \|\nabla(\partial_{j} u)\|_{L^{p}(\mathbb{R}^{n})} + C(n,p,\lambda,\Lambda),
\end{align*}
and hence
\begin{equation}
\|\nabla^{2}u\|_{L^{p}(\mathbb{R}^{n})}\le C(n,p,\lambda,\Lambda).\label{eq:Lq-hessian-bound}
\end{equation}
Hence, it follows from the Sobolev embedding that $\|\nabla u\|_{L^\infty(\R^n)}\le C(n,p,\lambda,\Lambda)$.

Now we would like to derive the $L^{2}$-decay of $\nabla u$. Combining \eqref{eq:Lp-gradient-bound} and \eqref{eq:Lq-hessian-bound}, it is easy to see that 
\begin{align*}
 & \int_{\mathbb{R}^{n}}e^{\frac{\lambda}{2}|{\bf x}|}|\partial_{j}u|^{2}\,\mathsf{d}{\bf x}=\int_{\mathbb{R}^{n}}e^{\frac{\lambda}{2}|{\bf x}|}(\partial_{j}u)(\partial_{j}u)\,\mathsf{d}{\bf x}\\
 & =-\frac{\lambda}{2}\int_{\mathbb{R}^{n}}\frac{x_{j}}{|{\bf x}|}e^{\frac{\lambda}{2}|{\bf x}|}u\partial_{j}u\,\mathsf{d}{\bf x}-\int_{\mathbb{R}^{n}}e^{\frac{\lambda}{2}|{\bf x}|}(\partial_{j}^{2}u)u\,\mathsf{d}{\bf x}\\
 & \le\frac{\lambda}{2}\int_{\mathbb{R}^{n}}e^{\frac{\lambda}{2}|{\bf x}|}|u||\partial_{j}u|\,\mathsf{d}{\bf x}+\int_{\mathbb{R}^{n}}e^{\frac{\lambda}{2}|{\bf x}|}|u||\partial_{j}^{2}u|\,\mathsf{d}{\bf x}\\
 & \le\frac{\lambda\Lambda}{2}\int_{\mathbb{R}^{n}}e^{-\frac{\lambda}{2}|{\bf x}|}|\partial_{j}u|\,\mathsf{d}{\bf x}+\Lambda\int_{\mathbb{R}^{n}}e^{-\frac{\lambda}{2}|{\bf x}|}|\partial_{j}^{2}u|\,\mathsf{d}{\bf x}\quad\text{(by \eqref{eq:decay-assumption})}\\
 & \le\frac{\lambda\Lambda}{2}\bigg(\int_{\mathbb{R}^{n}}e^{-\frac{\lambda}{2}p'|{\bf x}|}\,\mathsf{d}{\bf x}\bigg)^{\frac{1}{p'}}\|\partial_{j}u\|_{L^{p}(\mathbb{R}^{n})} + \Lambda\bigg(\int_{\mathbb{R}^{n}}e^{-\frac{\lambda}{2}p'|{\bf x}|}\,\mathsf{d}{\bf x}\bigg)^{\frac{1}{p'}}\|\partial_{j}^{2}u\|_{L^{p}(\mathbb{R}^{n})}\\
 & \le C(n,p,\lambda,\Lambda) \quad(\mbox{where $p'$ is the conjugate exponent of $p$}),
\end{align*}
that is, we obtain 
\begin{equation}
\int_{\mathbb{R}^{n}}e^{\frac{\lambda}{2}|{\bf x}|}|\nabla u|^{2}\,\mathsf{d}{\bf x}\le C(n,p,\lambda,\Lambda).\label{eq:L2-decay-gradient}
\end{equation}

We now continue to obtain the $L^{2}$-decay of $(-\Delta)^{\frac{1}{2}}u$. In view of \eqref{eq:Schrodinger} and using \eqref{eq:assumption-1}, \eqref{eq:assumption-2}, \eqref{eq:decay-assumption}, \eqref{eq:L2-decay-gradient}, we have
\begin{align}
 & \int_{\mathbb{R}^{n}}e^{\frac{\lambda}{2}|{\bf x}|}|(-\Delta)^{\frac{1}{2}}u|^{2}\,\mathsf{d}{\bf x}\nonumber \\
 & \le\int_{\mathbb{R}^{n}}e^{\frac{\lambda}{2}|{\bf x}|}|{\bf b}({\bf x})\cdot\nabla u|^{2}\,\mathsf{d}{\bf x}+\int_{\mathbb{R}^{n}}e^{\frac{\lambda}{2}|{\bf x}|}|q({\bf x})u|^{2}\,\mathsf{d}{\bf x}\nonumber \\
 & \le \|{\bf b}\|_{L^{\infty}(\mathbb{R}^{n})}^{2}\int_{\mathbb{R}^{n}}e^{\frac{\lambda}{2}|{\bf x}|}|\nabla u|^{2}\,\mathsf{d}{\bf x}+\|q\|_{L^{\infty}(\mathbb{R}^{n})}^{2}\int_{\mathbb{R}^{n}}e^{\frac{\lambda}{2}|{\bf x}|}|u|^{2}\,\mathsf{d}{\bf x}\nonumber \\
 & \le C(n,\lambda,\Lambda).\label{eq:L2-decay-half-Lap}
\end{align}
Here we may choose a smaller $\epsilon$ if necessary. 

Our next task is to derive the $L^{p}$-decay of $\nabla u$. First of all, let $p$ be odd. We then have
\begin{align}
 & \int_{\mathbb{R}^{n}}e^{\frac{\lambda}{2}|{\bf x}|}|\partial_{j}u|^{p}\,\mathsf{d}{\bf x} = \int_{\mathbb{R}^{n}}e^{\frac{\lambda}{2}|{\bf x}|}|\partial_{j}u|(\partial_{j}u)^{p-1}\,\mathsf{d}{\bf x}\nonumber \\
 & = \int_{\{\partial_{j}u\neq0\}}e^{\frac{\lambda}{2}|{\bf x}|}|\partial_{j}u|(\partial_{j}u)^{p-2}(\partial_{j}u)\,\mathsf{d}{\bf x}\nonumber \\
 & =-\frac{\lambda}{2}\int_{\mathbb{R}^{n}}\frac{x_{j}}{|{\bf x}|}e^{\frac{\lambda}{2}|{\bf x}|}|\partial_{j}u|(\partial_{j}u)^{p-2}u\,\mathsf{d}{\bf x} - \int_{\{\partial_{j}u\neq0\}} e^{\frac{\lambda}{2}|{\bf x}|}\frac{\partial_{j}u}{|\partial_{j}u|}(\partial_{j}^{2}u)(\partial_{j}u)^{p-2}u\,\mathsf{d}{\bf x}\nonumber \\
 & \quad -(p-2) \int_{\mathbb{R}^{n}}e^{\frac{\lambda}{2}|{\bf x}|}|\partial_{j}u|(\partial_{j}u)^{p-3}(\partial_{j}^{2}u)u\,\mathsf{d}{\bf x}\nonumber \\
 & \le\frac{\lambda\Lambda}{2}\int_{\mathbb{R}^{n}}e^{-\frac{\lambda}{2}|{\bf x}|}|\partial_{j}u|^{p-1}\,\mathsf{d}{\bf x} + \Lambda(p-1)\int_{\mathbb{R}^{n}}e^{-\frac{\lambda}{2}|{\bf x}|}|\partial_{j}u|^{p-2}|\partial_{j}^{2}u|\,\mathsf{d}{\bf x}\quad\text{(by \eqref{eq:decay-assumption})}\nonumber \\
 & \le C(n,p,\lambda,\Lambda) + \Lambda(p-1) \int_{\mathbb{R}^{n}}e^{-\frac{\lambda}{2}|{\bf x}|}|\partial_{j}u|^{p-2}|\partial_{j}^{2}u|\,\mathsf{d}{\bf x}\quad\text{(using \eqref{eq:Lp-gradient-bound})}.\label{eq:Lr-decay-gradient1}
\end{align}
Note that 
\begin{align*}
 & \int_{\mathbb{R}^{n}}e^{-\frac{\lambda}{2}|{\bf x}|}|\partial_{j}u|^{p-2}|\partial_{j}^{2}u|\,\mathsf{d}{\bf x}\\
 & \le\bigg(\int_{\mathbb{R}^{n}}e^{-\frac{\lambda}{2}r_{1}|{\bf x}|}\,\mathsf{d}{\bf x}\bigg)^{\frac{1}{r_{1}}}\bigg(\int_{\mathbb{R}^{n}}|\partial_{j}u|^{r_{2}(p-2)}\,\mathsf{d}{\bf x}\bigg)^{\frac{1}{r_{2}}}\bigg(\int_{\mathbb{R}^{n}}|\partial_{j}^{2}u|^{r_{3}}\,\mathsf{d}{\bf x}\bigg)^{\frac{1}{r_{3}}},
\end{align*}
where $1<r_{1},r_{2},r_{3}<\infty$ satisfy 
\[
\frac{1}{r_{1}}+\frac{1}{r_{2}}+\frac{1}{r_{3}}=1.
\]
Since we consider odd $p \ge 3$, we can choose $r_{1}=p$, $r_{2}=\frac{p}{p-2}$ and $r_{3}=p$. Hence, we obtain from \eqref{eq:Lp-gradient-bound} and \eqref{eq:Lq-hessian-bound} that
\begin{equation}
\int_{\mathbb{R}^{n}}e^{-\frac{\lambda}{2}|{\bf x}|}|\partial_{j}u|^{p-2}|\partial_{j}^{2}u|\,\mathsf{d}{\bf x}\le C(n,p,\lambda,\Lambda).\label{eq:Lr-decay-gradient2}
\end{equation}
Combining \eqref{eq:Lr-decay-gradient1} and \eqref{eq:Lr-decay-gradient2} gives
\begin{equation}
\int_{\mathbb{R}^{n}}e^{\frac{\lambda}{2}|{\bf x}|}|\nabla u|^{p}\,\mathsf{d}{\bf x}\le C(n,p,\lambda,\Lambda).\label{eq:Lr-decay-gradient}
\end{equation}
When $p$ is even, estimate \eqref{eq:Lr-decay-gradient} follows from the same argument above by noting $|\partial_ju|^{p}=(\partial_ju)^{p}$.

Finally, we estimate the $L^{p}$-decay of $(-\Delta)^{\frac{1}{2}}u$. Using the equation \eqref{eq:Schrodinger} and by \eqref{eq:assumption-1}, \eqref{eq:assumption-2}, \eqref{eq:decay-assumption}, \eqref{eq:Lr-decay-gradient}), we have 
\begin{align}
 & \int_{\mathbb{R}^{n}}e^{\frac{\lambda}{2}|{\bf x}|}|(-\Delta)^{\frac{1}{2}}u|^{p}\,\mathsf{d}{\bf x}\nonumber \\
 & \le C\bigg(\int_{\mathbb{R}^{n}}e^{\frac{\lambda}{2}|{\bf x}|}|{\bf b}({\bf x})\cdot\nabla u|^{p}\,\mathsf{d}{\bf x} + \int_{\mathbb{R}^{n}}e^{\frac{\lambda}{2}|{\bf x}|}|q({\bf x})u|^{p}\,\mathsf{d}{\bf x}\bigg)\nonumber \\
 & \le C\bigg( \|{\bf b}\|_{L^{\infty}(\mathbb{R}^{n})}^{p} \int_{\mathbb{R}^{n}}e^{\frac{\lambda}{2}|{\bf x}|}|\nabla u|^{p}\,\mathsf{d}{\bf x} + \|q\|_{L^{\infty}(\mathbb{R}^{n})}^{p} \int_{\mathbb{R}^{n}}e^{\frac{\lambda}{2}|{\bf x}|}|u|^{p}\,\mathsf{d}{\bf x}\bigg)\nonumber \\
 & \le C(n,p,\lambda,\Lambda).\label{eq:Lr-decay-half-Lap}
\end{align}
Consequently, \eqref{eq:gradient-decay} is a direct consequence of \eqref{eq:L2-decay-half-Lap} and \eqref{eq:Lr-decay-half-Lap}.
\end{proof}

\section{\label{sec:Caffarelli-Silvestre-extension}Caffarelli-Silvestre extension}

In this section, we briefly discuss the Caffarelli-Silvestre extension \cite{CS07extension}. We also refer to \cite[Appendix~A]{GR19unique} for higher order fractional Laplacian.

Let $\mathbb{R}_{+}^{n+1}:=\mathbb{R}^{n}\times\mathbb{R}_{+}=\begin{Bmatrix}\begin{array}{l|l}
{\bf x}=({\bf x}',x_{n+1}) & {\bf x}'\in\mathbb{R}^{n},x_{n+1}>0\end{array}\end{Bmatrix}$ and ${\bf x}_{0}=({\bf x}',0)\in\mathbb{R}^{n}\times\{0\}$. For
$R>0$, we denote 
\begin{align*}
B_{R}^{+}({\bf x}_{0}) & :=\begin{Bmatrix}\begin{array}{l|l}
{\bf x}\in\mathbb{R}_{+}^{n+1} & |{\bf x}-{\bf x}_{0}|\le R\end{array}\end{Bmatrix},\\
B_{R}'({\bf x}_{0}) & :=\begin{Bmatrix}\begin{array}{l|l}
{\bf x}\in\mathbb{R}^{n}\times\{0\} & |{\bf x}-{\bf x}_{0}|\le R\end{array}\end{Bmatrix}.
\end{align*}
To simplify the notations, we also denote $B_{R}^{+}:=B_{R}^{+}(0)$
and $B_{R}':=B_{R}'(0)$. We define two Sobolev spaces
\begin{align*}
\dot{H}^{1}(\mathbb{R}_{+}^{n+1}) & :=\begin{Bmatrix}\begin{array}{l|l}
v:\mathbb{R}_{+}^{n+1}\rightarrow\mathbb{R} & \int_{\mathbb{R}_{+}^{n+1}}|\nabla v|^{2}\,\mathsf{d}{\bf x}<\infty\end{array}\end{Bmatrix},\\
H_{{\rm loc}}^{1}(\mathbb{R}_{+}^{n+1}) & :=\begin{Bmatrix}\begin{array}{l|l}
v\in\dot{H}^{1}(\mathbb{R}_{+}^{n+1}) & \int_{\mathbb{R}^{n}\times(0,r)}|v|^{2}\,\mathsf{d}{\bf x}<\infty\text{ for some constant }r>0\end{array}\end{Bmatrix}.
\end{align*}
 
Given any $\mu \in \mathbb{R}$ and $u\in H^{\mu}(\mathbb{R}^{n})$. Following from \cite[Lemma~A.1]{GR19unique}, there exists $\tilde{u} \in \mathcal{C}^{\infty}(\mathbb{R}_{+}^{n+1})$ such that
\begin{equation}
\begin{cases}
\Delta\tilde{u}=0  \quad \text{in}\;\;\,\mathbb{R}_{+}^{n+1}, \\
\displaystyle{\lim_{x_{n+1} \rightarrow 0}} \| \tilde{u}(\cdot,x_{n+1}) - u \|_{H^{\mu}(\mathbb{R}^{n})} = 0, 
\end{cases}\label{eq:extension-problem1}
\end{equation}
where $\nabla=(\nabla',\partial_{n+1})=(\partial_{1},\cdots,\partial_{n},\partial_{n+1})$, and the half Laplacian is equivalent to the Dirichlet-to-Neumann map of the extension
problem \eqref{eq:extension-problem1}: 
\begin{equation}
\lim_{x_{n+1}\rightarrow0} \|\partial_{n+1}\tilde{u}(\cdot ,x_{n+1}) +(-\Delta)^{\frac{1}{2}}u \|_{H^{\mu - 1}(\mathbb{R}^{n})} =0 \label{eq:fractional-Lap-DNmap}
\end{equation}
(see \cite[(A.3)]{GR19unique}). In particular, when $\mu = \frac{1}{2}$, it follows from \cite[Corollary~A.2]{GR19unique} that 
\[
\| \tilde{u} \|_{\dot{H}^{1}(\mathbb{R}_{+}^{n+1})} \le C \| u \|_{H^{\frac{1}{2}}(\mathbb{R}^{n})} \quad \text{for some positive constant }C.
\]
In view of this observation, if $u \in H^{1}(\mathbb{R}^{n})$ and both ${\bf b},q$ are bounded, we can reformulate \eqref{eq:Schrodinger}
as the following local elliptic equation: 
\begin{equation}
\begin{cases}
\Delta\tilde{u}=0 & \text{in}\;\;\mathbb{R}_{+}^{n+1},\\
\tilde{u}({\bf x}',0)=u({\bf x}') & \text{on}\;\;\mathbb{R}^{n} \quad (\text{in}\;\; H^{1}(\mathbb{R}^{n})\text{-sense}),\\
{\displaystyle \lim_{x_{n+1}\rightarrow0}}\partial_{n+1}\tilde{u}({\bf x})={\bf b}({\bf x}')\cdot\nabla'u+q({\bf x}')u & \text{on}\;\;\mathbb{R}^{n} \quad (\text{in}\;\; L^{2}(\mathbb{R}^{n})\text{-sense}).
\end{cases}\label{eq:sch-localized}
\end{equation}
Since $u \in H^{1}(\mathbb{R}^{n}) \equiv {\rm dom} \, ((-\Delta)^{\frac{1}{2}})$, from \cite[page 48--49]{Sti10}, we have that $\tilde{u}\in H_{{\rm loc}}^{1}(\mathbb{R}_{+}^{n+1})$ and 
\begin{equation}
	\|\tilde{u}(\bullet,x_{n+1})\|_{L^{2}(\mathbb{R}^{n})}\le\|u\|_{L^{2}(\mathbb{R}^{n})}.\label{eq:extension-problem-est}
\end{equation}

\section{\label{sec:Some-estimates-CS}Some estimates related to the extension
problem }

The following lemma is a special case of \cite[Equation~(19)]{RW19Landis} (see also \cite[Lemma~3.2]{KWLandisDrift}). 
\begin{lem}
Let $\tilde{u}\in \mathcal{C}^{\infty}(\mathbb{R}_{+}^{n+1})$ be a solution
to \eqref{eq:extension-problem1}. Then the following estimate holds
for any ${\bf x}_{0}\in\mathbb{R}^{n}\times\{0\}$: 
\begin{align}
\|\tilde{u}\|_{L^{2}(B_{{c}R}^{+}({\bf x}_0))}\le& C\bigg(\|\tilde{u}\|_{L^{2}(B_{16R}^{+}({\bf x}_0))}+R^{\frac{1}{2}}\|u\|_{L^{2}(B_{16R}'({\bf x}_0))}\bigg)^{\alpha}\nonumber\\
&\times\bigg(R^{\frac 32}\bigg\|\lim_{x_{n+1}\rightarrow0}\partial_{n+1}\tilde{u}\bigg\|_{L^{2}(B_{16R}'({\bf x}_0))}+R^{\frac 12}\|u\|_{L^{2}(B_{16R}'({\bf x}_0))}\bigg)^{1-\alpha}\label{eq:recall-RW19-eqn19}
\end{align}
for some positive constants $C=C(n)$, $\alpha=\alpha(n)\in(0,1)$
and ${c}={c}(n)\in(0,1)$, all of them are independent of $R$ and ${\bf x}_{0}$. 
\end{lem}

By choosing $\sigma=\frac{1}{2}$, $\nu=2$ and $a({\bf x}')\equiv0$
in \cite[Proposition 2.6(i)]{JLX14}, we obtain the following version
of De Giorgi-Nash-Moser type theorem. 
\begin{lem}
Let $\tilde{u}\in \mathcal{C}^{\infty}(\mathbb{R}_{+}^{n+1})$ be satisfy \eqref{eq:extension-problem1} and $p>n$. There exists a
constant $C=C(n,p)>0$ such that 
\begin{equation}
\|\tilde{u}\|_{L^{\infty}(B_{\frac{1}{4}}^{+})}\le C\bigg[\|\tilde{u}\|_{L^{2}(B_{1}^{+})}+\bigg\|\lim_{x_{n+1}\rightarrow0}\partial_{n+1}\tilde{u}\bigg\|_{L^{p}(B_{1}')}\bigg].\label{eq:Schauder-estimate}
\end{equation}
\end{lem}

Combining \eqref{eq:recall-RW19-eqn19} and \eqref{eq:Schauder-estimate},
together with some suitable scaling, we can obtain the following lemma. 
\begin{lem}
Let $\tilde{u}\in \mathcal{C}^{\infty}(\mathbb{R}_{+}^{n+1})$ be a solution to
 \eqref{eq:extension-problem1} and $p>n$. Then the following inequality
holds for all ${\bf x}_{0}\in\mathbb{R}^{n}\times\{0\}$ and $R\ge1$:
\begin{align}
\|\tilde{u}\|_{L^{\infty}(B_{cR}^{+}({\bf x}_{0}))}& \le C\bigg(\|\tilde{u}\|_{L^{2}(B_{16R}^{+}({\bf x}_{0}))}+R^{\frac{1}{2}}\|u\|_{L^{2}(B_{16R}'({\bf x}_{0}))}\bigg)^{\alpha}\nonumber \\
 & \quad\times\bigg(R^{\frac{3}{2}}\|(-\Delta)^{\frac{1}{2}}u\|_{L^{2}(B_{16R}'({\bf x}_{0}))}+R^{\frac{1}{2}}\|u\|_{L^{2}(B_{16R}'({\bf x}_{0}))}\bigg)^{1-\alpha}\nonumber \\
 & \quad+R^{\frac{3}{2}}\|(-\Delta)^{\frac{1}{2}}u\|_{L^{p}(B_{R}'({\bf x}_{0}))}\label{eq:boundary-bulk1}
\end{align}
for some positive constants $C=C(n,p)$, $\alpha=\alpha(n)\in(0,1)$
and ${c}={c}(n)\in(0,1)$, all of them are independent
of $R$ and ${\bf x}_{0}$. 
\end{lem}

\begin{proof}
Without loss of generality, it suffices to take ${\bf x}_{0}=0$. Let $\tilde{v}({\bf x})=\tilde{u}(R{\bf x})$ 
and let $v({\bf x}')=u(R{\bf x}')$, we observe that 
\[
\begin{cases}
\Delta\tilde{v}=0 & \text{in}\;\;\mathbb{R}_{+}^{n+1},\\
\tilde{v}({\bf x}',0)=v({\bf x}') & \text{on}\;\;{\bf x}'\in\mathbb{R}^{n}.
\end{cases}
\]
From \eqref{eq:recall-RW19-eqn19} and \eqref{eq:Schauder-estimate}, it follows that 
\begin{align}
\|\tilde{v}\|_{L^{\infty}(B_{c}^{+})}& \le C\bigg(\|\tilde{v}\|_{L^{2}(B_{16}^{+})}+\|v\|_{L^{2}(B_{16}')}\bigg)^{\alpha}\bigg(\bigg\|\lim_{x_{n+1}\rightarrow0}\partial_{n+1}\tilde{v}\bigg\|_{L^{2}(B_{16}')}+\|v\|_{L^{2}(B_{16}')}\bigg)^{1-\alpha}\nonumber \\
 & \quad+\bigg\|\lim_{x_{n+1}\rightarrow0}\partial_{n+1}\tilde{v}\bigg\|_{L^{p}(B_{1}')}.\label{eq:Schauder2}
\end{align}
Note that 
\begin{align*}
\|\tilde{v}\|_{L^{\infty}(B_{c}^{+})} & =\|\tilde{u}\|_{L^{\infty}(B_{cR}^{+})},\\
\|\tilde{v}\|_{L^{2}(B_{16}^{+})} & =R^{-\frac{n+1}{2}}\|\tilde{u}\|_{L^{2}(B_{16R}^{+})},\\
\|v\|_{L^{2}(B_{16}')} & =R^{-\frac{n}{2}}\|u\|_{L^{2}(B_{16R}')},\\
\bigg\|\lim_{x_{n+1}\rightarrow0}\partial_{n+1}\tilde{v}\bigg\|_{L^{p}(B_{1}')} & =R^{1-\frac{n}{p}}\bigg\|\lim_{x_{n+1}\rightarrow0}\partial_{n+1}\tilde{u}\bigg\|_{L^{p}(B_{R}')},\quad p\ge2.
\end{align*}
Hence, \eqref{eq:Schauder2} becomes 
\begin{align*}
\|\tilde{u}\|_{L^{\infty}(B_{cR}^{+})}& \le CR^{-\frac{n}{2}}R^{-\frac{\alpha}{2}}\bigg(\|\tilde{u}\|_{L^{2}(B_{16R}^{+})}+R^{\frac{1}{2}}\|u\|_{L^{2}(B_{16R}')}\bigg)^{\alpha}\\
 & \quad\times R^{-\frac{1}{2}(1-\alpha)}\bigg(R^{\frac{3}{2}}\bigg\|\lim_{x_{n+1}\rightarrow0}\partial_{n+1}\tilde{u}\bigg\|_{L^{2}(B_{16R}')}+R^{\frac{1}{2}}\|u\|_{L^{2}(B_{16R}')}\bigg)^{1-\alpha}\\
 & \quad+R^{-\frac{1}{2}-\frac{n}{p}}R^{\frac{3}{2}}\bigg\|\lim_{x_{n+1}\rightarrow0}\partial_{n+1}\tilde{u}\bigg\|_{L^{p}(B_{R}')}.
\end{align*}
Since $R\ge1$, \eqref{eq:boundary-bulk1} follows immediately. 
\end{proof}

\section{\label{sec:Boundary-Bulk}Boundary decay to bulk decay}

In this section, we will establish that the boundary decay implies the bulk decay.   
\begin{prop}
\label{prop:boundary-bulk}Assume that \eqref{eq:assumption-1} and \eqref{eq:assumption-2} are satisfied. 
Let $u \in W^{2,p}(\mathbb{R}^{n})$ for some integer $n < p <\infty$ be a solution to \eqref{eq:Schrodinger} and the decay assumption \eqref{eq:decay-assumption} holds. Then 
\begin{equation}
|\tilde{u}({\bf x})|\le Ce^{-c|{\bf x}|}\quad\text{for }\;\;{\bf x}=({\bf x}',x_{n+1})\in\mathbb{R}_{+}^{n+1}.\label{eq:bulk-decay}
\end{equation}
\end{prop}

\begin{proof}
Given any $R\ge1$, choosing ${\bf x}_{0}\in\mathbb{R}^{n}\times\{0\}$
with $|{\bf x}_{0}|=32R$. By \eqref{eq:decay-assumption}, we have
\begin{equation}
\|u\|_{L^{2}(B_{16R}'({\bf x}_{0}))}\le Ce^{-cR}.\label{eq:boundary-to-bulk1}
\end{equation}
Furthermore, \eqref{eq:extension-problem-est} yields
\begin{equation}
\|\tilde{u}\|_{L^{2}(B_{16R}^{+}({\bf x}_{0}))}\le\|\tilde{u}\|_{L^{2}(\mathbb{R}^{n}\times(0,16R))}\le4R^{\frac{1}{2}}\|u\|_{L^{2}(\mathbb{R}^{n})}\le C(n,\lambda,\Lambda)R^{\frac{1}{2}}.\label{eq:boundary-to-bulk2}
\end{equation}
Plugging \eqref{eq:gradient-decay}, \eqref{eq:boundary-to-bulk1}
and \eqref{eq:boundary-to-bulk2} into \eqref{eq:boundary-bulk1} implies 
\[
\|\tilde{u}\|_{L^{\infty}(B_{cR}^{+}({\bf x}_{0}))}\le Ce^{-cR}.
\]
Following the chain of balls argument described in \cite[Proposition 2.2, Step 2]{RW19Landis},
we finally conclude our result. 
\end{proof}

\section{\label{sec:main}Proof of Theorem~{\rm \ref{thm:main1}}}

Recall the following Liouville-type theorem in \cite[Theorem B]{Armitage85Liouville}. 
\begin{thm}
\label{thm:Liouville}Suppose that $\Delta\tilde{u}=0$
in $\mathbb{R}_{+}^{n+1}$. If $\tilde{u}$ satisfies the decay property
\eqref{eq:bulk-decay}, then $\tilde{u}\equiv0$. 
\end{thm}
\noindent It is obvious that Theorem~\ref{thm:main1} is an easy consequence of Proposition~\ref{prop:boundary-bulk} and Theorem~\ref{thm:Liouville}. We now say a few words about the proof of Theorem~\ref{thm:main2}. As Proposition~\ref{prop:boundary-bulk}, the boundary decay \eqref{cx} implies the bulk decay \eqref{eq:bulk-decay}. In the case of ${\bf b}\equiv 0$, the proof of Proposition~\ref{prop:boundary-bulk} remains true when $q$ is bounded. 

To make the paper self-contained, we will give another proof of Theorem~\ref{thm:Liouville} in Appendix~\ref{sec:Appendix1}.

\appendix

\section{\label{sec:Appendix1}Proof of Theorem~\ref{thm:Liouville}}

First of all, we introduce a mapping from the ball to the upper half-space,
and back, which preserves the Laplacian. For convenience, we define
\[
{\bf x}^{*}:=\frac{{\bf x}}{|{\bf x}|^{2}}\quad\text{for }\,\,{\bf x}\in\mathbb{R}^{n+1}\setminus\{0\},
\]
see e.g. \cite{ABR01Harmonic}. Let ${\bf s}=(0,\cdots,0,-1)$
be the south pole of the unit sphere $\mathcal{S}^{n}$, and we define
\[
\Phi({\bf z}):=2({\bf z}-{\bf s})^{*}+{\bf s}=\frac{(2{\bf z}',1-|{\bf z}|^{2})}{|{\bf z}'|^{2}+(1+z_{n+1})^{2}}
\]
for all ${\bf z}=({\bf z}',z_{n+1})\in\mathbb{R}^{n+1}\setminus\{{\bf s}\}$.
It is easy to see that $\Phi^{2}={\rm Id}$. 

Let $B_{1}(0)$ be the unit ball in $\mathbb{R}^{n+1}$. The following
lemma can be found in \cite{ABR01Harmonic}. 
\begin{lem}
\label{lem:conformal}The mapping $\Phi:\mathbb{R}^{n+1}\setminus\{{\bf s}\}\rightarrow\mathbb{R}^{n+1}\setminus\{{\bf s}\}$
is injective. Furthermore, it maps $B_{1}(0)$ onto $\mathbb{R}_{+}^{n+1}$,
and maps $\mathbb{R}_{+}^{n+1}$ onto $B_{1}(0)$. It also maps $\mathcal{S}^{n}\setminus\{{\bf s}\}$ onto ${\mathbb R}^n$ and maps ${\mathbb R}^n$ onto $\mathcal{S}^{n}\setminus\{{\bf s}\}$.
\end{lem}

Given any function $w$ defined on a domain $\Omega$ in $\mathbb{R}^{n+1}\setminus\{{\bf s}\}$.
The Kelvin transform $\mathcal{K}[w]$ of $w$ is defined by 
\begin{equation}
\mathcal{K}[w]({\bf z}):=2^{\frac{n-1}{2}}|{\bf z}-{\bf s}|^{1-n}w(\Phi({\bf z}))\quad\text{for all}\,\,{\bf z}\in\Phi(\Omega).\label{eq:Kelvin}
\end{equation}
The following lemma can be found in \cite{ABR01Harmonic},
which exhibits a crucial property of the Kelvin transform. 
\begin{lem}
\label{lem:harmonic}Let $\Omega$ be any domain in $\mathbb{R}^{n+1}\setminus\{{\bf s}\}$.
Then $u$ is harmonic on $\Omega$ if and only if $\mathcal{K}[u]$
is harmonic on $\Phi(\Omega)$. 
\end{lem}

Now, we are ready to prove Theorem~\ref{thm:Liouville}. 
\begin{proof}
[Proof of Theorem~{\rm \ref{thm:Liouville}}] To begin, it is not hard to compute
\begin{align*}
|\Phi({\bf z})| & =\frac{\sqrt{4|{\bf z}'|^{2}+(|{\bf z}|^{2}-1)^{2}}}{|{\bf z}'|^{2}+(1+z_{n+1})^{2}}=\bigg|\frac{2|{\bf z}'|+i((-z_{n+1})^{2}+|{\bf z}'|^{2}-1)}{(-z_{n+1}-1)^{2}+|{\bf z}'|^{2}}\bigg|\\
 & =\bigg|\frac{(-z_{n+1}+1)+i|{\bf z}'|}{(-z_{n+1}-1)+i|{\bf z}'|}\bigg|.
\end{align*}
The decay assumption \eqref{eq:bulk-decay} implies that for ${\bf z}$ near the south pole ${\bf s}$,
\begin{align*}
|\mathcal{K}[\tilde{u}]({\bf z})| & =2^{\frac{n-1}{2}}|{\bf z}-{\bf s}|^{1-n}|\tilde{u}(\Phi({\bf z}))| \le C|{\bf z}-{\bf s}|^{1-n}e^{-c|\Phi({\bf z})|}\\
 & =C|{\bf z}-{\bf s}|^{1-n}\exp\bigg(-c\bigg|\frac{(-z_{n+1}+1)+i|{\bf z}'|}{(-z_{n+1}-1)+i|{\bf z}'|}\bigg|\bigg)\\
 & \le C|{\bf z}-{\bf s}|^{1-n}\exp\bigg(-c\frac{1}{|{\bf z}-{\bf s}|}\bigg)\\
 & \approx C\exp\bigg(-c\frac{1}{|{\bf z}-{\bf s}|}\bigg).
\end{align*}
From Lemma~\ref{lem:harmonic}, we know that $\mathcal{K}[\tilde{u}]$ is harmonic on $B_{1}(0)$.
By \cite[Theorem 1]{Jin93UCP}, we obtain that $\mathcal{K}[\tilde{u}]\equiv0$.
In view of \eqref{eq:Kelvin} and Lemma~\ref{lem:conformal}, we then conclude that
$\tilde{u}\equiv0$ in $\Phi(B_{1})=\mathbb{R}_{+}^{n+1}$.
\end{proof}

\section*{Acknowledgments}

Kow is partially supported by the Academy of Finland (Centre of Excellence in Inverse Modelling and Imaging, 312121) and by the European Research Council under Horizon 2020 (ERC CoG 770924). Wang is partially supported by MOST 108-2115-M-002-002-MY3 and MOST 109-2115-M-002-001-MY3. 

\end{sloppypar}

\bibliographystyle{custom}
\bibliography{ref}
\end{document}